\documentclass[11pt, a4paper]{amsart}
\usepackage{color}

\newtheorem{prop}{Proposition}[section]
\newtheorem{teo}{Theorem}[section]

\def\ep{\varepsilon}

\def\a{\mathfrak a}
\def\R{\mathbb R}
\def\K{{\mathcal K}}

\def\T{{\mathcal T}}

\begin{document}

\title[Nonlocal diffusion
with absorption]{Asymptotic behavior for a nonlocal diffusion equation with
absorption and nonintegrable initial data. The supercritical case}

\author[J. Terra \and N. Wolanski]{Joana Terra\and Noemi Wolanski}

\thanks{
\noindent 2000 {\it Mathematics Subject Classification } 35K57,
35B40.}

\keywords{Nonlocal diffusion, boundary value problems.}
\address{Joana Terra\hfill\break\indent
Departamento  de Matem{\'a}tica, FCEyN \hfill\break\indent UBA (1428)
Buenos Aires, Argentina.} \email{{\tt jterra@dm.uba.ar} }

\address{Noemi Wolanski \hfill\break\indent
Departamento  de Matem{\'a}tica, FCEyN \hfill\break\indent UBA (1428)
Buenos Aires, Argentina.} \email{{\tt wolanski@dm.uba.ar} }

\date{}

\begin{abstract} In this paper we study the asymptotic behavior as time goes to infinity of the solution to a nonlocal diffusion equation with absorption modeled by a powerlike reaction $-u^p$, $p>1$ and set in $\R^N$. We consider a bounded, nonnegative initial datum $u_0$ that behaves like a negative power at infinity. That is, $|x|^\alpha u_0(x)\to A>0$ as $|x|\to\infty$ with $0<\alpha\le N$. We prove that, in the supercritical case $p>1+2/\alpha$, the solution behaves asymptotically as that of the heat equation --with diffusivity $\a$ related to the nonlocal operator-- with the same initial datum.

\end{abstract}

\maketitle

\date{}

\section{Introduction}
\label{Intro} \setcounter{equation}{0}
In this paper we study the asymptotic behavior as time goes to infinity of the solution to a nonlocal diffusion equation in $R^N$ with an absorption term. We are interested in the case in which the initial datum is nonintegrable. More precisely, we consider a bounded datum $u_0(x)$ such that $|x|^\alpha u_0(x)\to A>0$ as $|x|\to\infty$ with $0<\alpha\le N$. The problem under consideration is the following
\begin{equation}
\begin{aligned}\label{problem}
&u_t(x,t)=\int J(x-y)\big(u(y,t)-u(x,t)\big)\,dy - u^p(x,t)\\
&u(x,0)=u_0(x)
\end{aligned}
\end{equation}
where $J\in C_0^\infty(\R^N)$, radially symmetric with $J\ge0$ and $\int J=1$.

The term $\int J(x-y)\big(u(y,t)-u(x,t)\big)\,dy$ in \eqref{problem} represents diffusion. In fact, let $u(x,t)$ be the density of a population at the point $x$ at time $t$. The kernel $J(x-y)$ may be seen as the probability distribution density of jumping from site $y$ to site $x$. Thus, $\int J(x-y)u(y,t)\,dy$ represents the rate at which the population is arriving at the site $x$ from all over space. By the symmetry of the kernel $J$, $\int J(x-y)u(x,t)\,dy$ represents then the rate at which it is  leaving the point $x$.

The absorption term $-u^p(x,t)$ in the equation represents a rate of
consumption due to an internal reaction.

 This diffusion operator has been used to model several nonlocal diffusion processes in the last few years. See for instance \cite{BCh1,BCh2,BFRW,CF,F,Z}. In particular, nonlocal diffusions are of interest in biological and biomedical problems. Recently, this kind of nonlocal operators  have also been used  for image enhancement  \cite{GO}.

\bigskip

It can be seen that -when properly rescaled- the nonlocal operator approximates $\a
\Delta$, where $\Delta$ is the Laplace operator and $\a$ is a constant that depends
only on $J$ and the dimension $N$. (See, for instance, \cite{CER, CERW2}).

Another close relation to the heat operator with diffusivity $\a$ was found in \cite{ChChR, IR}
where it was proven that, for bounded and integrable initial data, the asymptotic
behavior as $t$ tends to infinity of the solution $u_L$ to the equation without
absorption 
\begin{equation}\label{probl-homo} u_t(x,t)=\int J(x-y)\big(u(y,t)-u(x,t)\big)\,dy =Lu,
\end{equation}
 is the same
as that of the solution of the heat equation with diffusivity $\a$ and the same initial condition.
Namely,
\begin{equation}\label{asym-L1}
t^{N/2}\|u(x,t)-MU_\a(x,t)\|_{L^\infty(|x|\le K\sqrt t)}\to0\quad\mbox{as }t\to\infty\quad\forall K>0,
\end{equation}
with $U_\a$ the fundamental solution of the heat equation with diffusivity $\a$ and $M=\int u_0$
the  initial mass. The assumption on the initial condition is that $u_0\in L^\infty\cap L^1$.

In a recent paper \cite{PR}, the nonlocal equation with absorption \eqref{problem} was considered in the case $u(x,0)\in L^\infty\cap L^1$ and $p>1+2/N$. The authors prove
that $u$ also satisfies \eqref{asym-L1}. Some
results for $p<1+2/N$ can also be found in that paper. The critical case $p=1+2/N$
was left open.

\bigskip

On the other hand, the asymptotic behavior of the solutions to the heat equation and of the heat
equation with absorption are very well known for some nonintegrable initial data with a precise
power like behavior at infinity. In fact, Kamin and Peletier in \cite{KP} study the
asymptotic behavior of the solution to
\begin{equation}\label{heat}
\begin{aligned}
&u_t=\a\Delta u-u^p\qquad\mbox{in }\R^N\times(0,\infty),\\
&u(x,0)=u_0(x)\qquad\mbox{in }\R^N,
\end{aligned}
\end{equation}
with $u_0\in L^\infty$ and $u_0\ge0$ such that 
\begin{equation}\label{condalpha}
|x|^{\alpha} u_0(x)\to A>0 \quad\text{ as }\, |x|\to\infty,
\end{equation}
with $0<\alpha<N$. In fact, the authors prove that
\begin{equation}\label{asym-heat-alpha}
t^{\alpha/2}\|u(x,t)-U_{\alpha,A}(x,t)\|_{L^\infty(|x|\le K\sqrt t)}\to0\quad\mbox{as
}t\to\infty\quad\forall K>0,
\end{equation}
where $U_{\alpha,A}$ is the solution to
\begin{equation}\label{U-alpha}
\begin{aligned}
&U_t=\a\Delta U-\delta_{p,\alpha} U^p\qquad\mbox{in }\R^N\times(0,\infty),\\
&U(x,0)=\frac A{|x|^\alpha}\qquad\mbox{in }\R^N.
\end{aligned}
\end{equation}

Here $\delta_{p,\alpha}=0$ in the supercritical case $p>1+2/\alpha$ and $\delta_{p,\alpha}=1$ in the critical case $p=1+2/\alpha$.

Then, in \cite{KU}, Kamin and Ughi study the case where $\alpha=N$ in
\eqref{condalpha} and $p>1+2/N$. They prove that
\begin{equation}\label{heat-N}
t^{N/2}\Big\|\frac{u(x,t)}{\log t}- C_{A,N} U_\a(x,t)\Big\|_{L^\infty(|x|\le K\sqrt
t)}\to0\quad\mbox{as }t\to\infty\quad\forall K>0,
\end{equation}
with $C_{A,N}$ a constant that depends only on $A$ and $N$ and $U_\a$ the fundamental solution of the heat operator with diffusivity $\a$.

Later on, in \cite{H}, Herraiz fully analyzes problem \eqref{heat} and presents a
complete characterization including the asymptotic behavior outside the parabolas $|x|\le K\sqrt
t$.

\bigskip

It is the purpose of this paper to prove that, for initial data
satisfying \eqref{condalpha} with $0<\alpha\leq N$ and $p>1+2/\alpha$, the
asymptotic behavior of the solution to the nonlocal diffusion problem with
absorption \eqref{problem} is the same as the asymptotic behavior of solutions to
the heat equation with absorption  \eqref{heat}. That is, that
\eqref{asym-heat-alpha} holds in the case $\alpha<N$ and \eqref{heat-N} holds in the case
$\alpha=N$ with $u$ the solution of \eqref{problem}.

The proofs of these results rely strongly on estimates for the fundamental
solution of the equation without absorption \eqref{probl-homo} .

In \cite{ChChR} the authors observed that the fundamental solution of the operator in \eqref{probl-homo} can be written as
\begin{equation}\label{fund-sol}
U(x,t)=e^{-t}\delta +W(x,t),
\end{equation}
where $\delta$ is the Dirac measure and $W(x,t)$ is a smooth function.

Then, in \cite{IR} the authors established $L^{\infty}$ estimates for $W$ and
proved that
$$
t^{N/2}\|W(x,t)-U_\a(x,t)\|_{L^\infty(\R^N)}\to0\quad\mbox{as }t\to\infty,
$$
where $U_\a$ is the fundamental solution of the heat equation with diffusivity $\a$.

This estimate allowed them to prove that the solution of the homogeneous nonlocal diffusion equation \eqref{probl-homo} with bounded and integrable initial data behaves -for $t$ tending to infinity- as the solution of the heat equation with diffusivity $\a$, removing in this way the hypothesis that the Fourier transform of $u_0$ be in $L^1$ as was assumed in \cite{ChChR}.

In the present paper, as we are considering nonintegrable initial data, these $L^\infty$ estimates cannot be used. So, the first step is to get estimates in $L^q-$norms for $1\le q<\infty$.

\bigskip

\bigskip

As in \cite{PR}, the philosophy of the study of the asymptotic behavior for the problem with absorption in the
supercritical case is that, as time tends to infinity, the absorption term goes
to 0 so fast that the behavior is the same as that of the equation without absorption. Thus, what
we need to prove is essentially that, as time tends to infinity, the difference
between the solutions to both problems goes to 0 faster than each term separately.

In order to prove this result it is necessary to compare the solution $u$ of
\eqref{problem} with the solution $u_L$ of the  equation without absorption
\eqref{probl-homo} that coincides with $u$ at a large time $t_0$. Once we prove that
this difference is small the result follows if we know that the asymptotic behavior of $u_L$ is
the same as that of \eqref{heat}. To this end, we first analyze the asymptotic
behavior for the nonlocal equation without absorption \eqref{probl-homo} with an initial datum satisfying \eqref{condalpha}.
In order to finish the argument, it is essential to know, and we prove it in this paper, that the behavior in space of the solution of our problem \eqref{problem}
at any time $t_0$ is the same as that of the initial datum $u_0$. That is, that $u(\cdot, t_0)$ satisfies \eqref{condalpha} at any positive time $t_0$.

\bigskip

The techniques of this paper do not work in the critical case $p=1+2/\alpha$. In
fact, for such $p$, the asymptotic behavior is  that of  the heat equation with
absorption since the diffusion and the absorption terms are of the same order. We have to proceed
in a completely different way by going back to the techniques of the original paper for the heat
equation \cite{KP}, and we do so in \cite{TW}. The method in \cite{TW} also works in the
supercritical case considered in this paper. Nevertheless, the ideas in the
present paper are simpler and rather easy to implement and therefore, we consider
that it is worth presenting them separately.

\bigskip

The paper is organized as follows. In Section \ref{prelim} we prove the asymptotic estimates of
$W$ in $L^q$ norms for finite $q$ and use them to study the behavior of the
solution $u_L$ of the linear equation \eqref{probl-homo}
with a nonnegative, bounded initial datum satisfying \eqref{condalpha} with $0<\alpha\le N$.

In particular, we prove that the solution $u_L$ satisfies
\begin{equation}\label{asymp-homo}
t^{\alpha/2}\|u_L(x,t)-u_\Delta(x,t)\|_{L^\infty(\R^N)}\to0\qquad\mbox{as }t\to\infty,
\end{equation}
where $u_\Delta$ is the solution of the heat equation with diffusivity $\a$ and initial condition $u_L(x,0)$.

Moreover, in this section we establish the precise behavior of $u_L$ for
$|x|\to\infty$.

In Section \ref{existence} we prove the global existence and uniqueness of the solution to
\eqref{problem}  -with a more general absorption term $-|u|^{p-1}u$- for our
nonintegrable initial data with power like behavior at infinity. We also
establish that the solution is bounded, in any time interval [0,T], by $C_T(1+|x|)^{-\alpha}$ for a certain
constant $C_T$. Then, when $u_0$ is nonnegative, we deduce that the solution is nonnegative.
Finally, we prove that for every $t_0>0$
$$
|x|^\alpha u(x,t_0)
\to A\qquad\mbox{as }|x|\to\infty.
$$

Finally, in Section \ref{result} we prove our main result. That is, we prove that
\eqref{asym-heat-alpha} holds when $\alpha<N$ and \eqref{heat-N} when $\alpha=N$ with $u$ the
solution to \eqref{problem}.

\bigskip

\section{Preliminary results and homogeneous equation}\label{prelim}\setcounter{equation}{0}
In this section we study the asymptotic behavior of bounded solutions to the homogeneous equation
\begin{equation}\label{homo-sect2}
\begin{aligned}
&u_t=\int J(x-y)\big(u(y,t)-u(x,t)\big)\,dy\qquad\mbox{in }\R^N\times(0,\infty),\\
&u(x,0)=u_0(x)\qquad\mbox{in }\R^N.
\end{aligned}
\end{equation}

Existence  of bounded solutions for bounded initial data is a well know result
that  can be seen by different arguments. One possible construction is to convolve the initial
datum with the fundamental solution \eqref{fund-sol}. Uniqueness of bounded
solutions follows from the  comparison principle for bounded solutions. (See \cite{LW},
Proposition 2.2 with $\theta=0$ for the case of smooth solutions. In order to get the result in
our case we just have to approximate  the initial datum by smooth functions that are uniformly
bounded).

In this section we first obtain estimates in $L^q(\R^N)$ for the ``good'' part $W(x,t)$ of the fundamental solution of the nonlocal diffusion equation.

Then, we get the result on the asymptotic behavior as time goes to infinity of the solution of the equation for initial data that behave as a negative power at infinity.

Finally, we prove that the solution behaves as the initial datum for every $t>0$.

\bigskip

Let us recall (see \cite{ChChR}) that the fundamental solution of the equation \eqref{probl-homo} is given by
\begin{equation}\label{fundamental}
U(x,t)=e^{-t}\delta+W(x,t),
\end{equation}
where $W$ is a smooth function. In fact, by taking Fourier transform in the space variables the authors show that $W(\cdot,t)$ belongs to the Schwartz class of $C^\infty$ functions rapidly decaying at infinity together with their derivatives.

Moreover, in \cite{IR} by using the characterization of $W$ in terms of its Fourier transform, the authors prove that
\begin{equation}\label{aprox-W}
t^{N/2}\|W(x,t)-U_\a(x,t)\|_{L^\infty(\R^N)}\le C t^{-1/2},
\end{equation}
where $U_\a$ is the fundamental solution of the heat equation with diffusivity $\a$.

The Fourier transform method allows to get $L^q$ estimates only for  $q\ge2$. In order to get an $L^1$ estimate we observe that $W$ is the solution to
\begin{equation}\label{eq-W}
\begin{aligned}
&W_t(x,t)=\int J(x-y)\big(W(y,t)-W(x,t)\big)\,dy + e^{-t}J(x)\\
&W(x,0)=0.
\end{aligned}
\end{equation}

In fact, let us call $L$ the nonlocal operator. This is $Lu=J*u-u$. Then,
\begin{align*}
W_t(x,t)-LW(x,t)&=
U_t(x,t)-LU(x,t)-\frac\partial{\partial t}e^{-t}\delta
+e^{-t}L\delta(x)\\
&=e^{-t}\delta+e^{-t}\big(J*\delta-\delta\big)(x)\\
&= e^{-t}J(x).
\end{align*}

Moreover, since $U(x,0)=\delta$ there holds that $W(x,0)=0$.

In particular, by the comparison principle, $W\ge0$. On the other hand, since the unique bounded solution to the homogeneous nonlocal diffusion equation with initial datum $u_0\equiv1$ is the function $u(x,t)\equiv1$, the representation of this solutions gives
$$
1=u(x,t)=e^{-t}u_0(x)+\int W(x-y,t)u_0(y)\,dy= e^{-t}+\int W(y,t)\,dy.
$$
Thus, $\int W(y,t)\,dy=1-e^{-t}\le 1$ for every $t$.

Now, since $\int U_\a(y,t)\,dy=1$ for every $t$, interpolation with inequality \eqref{aprox-W} gives the estimate
\begin{equation}\label{estimate-W-U}
\|W(x,t)-U_\a(x,t)\|_{L^{q'}(\R^N)}\le C_q t^{-\frac{N+1}{2q}}\quad\mbox{if}\quad\frac1q+\frac1{q'}=1.
\end{equation}

In particular, for every $1\le q\le \infty$,
\begin{equation}\label{estimate-W}
\|W(x,t)\|_{L^{q'}(\R^N)}\le C_q t^{-\frac N{2q}}.
\end{equation}

\bigskip

With this estimate we can prove that for large $t$ the solution $u$ to
\eqref{homo-sect2} behaves as
the solution of the heat equation with diffusivity $\a$. In fact,

\begin{teo}\label{teo-u_L} Let $u_0\in L^\infty$, nonnegative such that there exist $0<\alpha\le N$
 and  $A>0$ with $|x|^\alpha u_0(x)\to A$ as $|x|\to\infty$. Let $u$ be the solution of the equation $u_t=Lu$ with initial condition $u_0$. Then, for every $0<\mu<\frac\alpha{2N}$ there exists a constant $C_\mu$ such that
\begin{equation}\label{aprox-homo}
t^{\alpha/2}\|u(x,t)-u_\Delta(x,t)\|_{L^\infty(\R^N)}
\le C_\mu t^{-\mu}.
\end{equation}

In particular,  if $\alpha<N$, due to the behavior of $u_\Delta$ (see, for instance \cite{KP}) there holds that $u(x,t)\le C t^{-\alpha/2}$ and
\begin{equation*}
t^{\alpha/2}\|u(x,t)-U_{\alpha,A}(x,t)\|_{L^\infty(|x|\le K\sqrt t)}\to0\quad\mbox{as }t\to\infty\quad\forall K>0,
\end{equation*}
where $U_{\alpha,A}$ is the solution to \eqref{U-alpha}. (Note that the results of \cite{KP,KU} apply also to the solution $u_\Delta$ of the homogeneous heat equation. In this case, the constant $\delta_{\alpha,A}$ in \eqref{U-alpha} is zero so that,  $U_{\alpha,A}$ is a solution to the homogeneous heat equation).
 
If $\alpha=N$, due to the results in \cite{KU} for the solution of the homogeneous heat equation, there holds that $u(x,t)\le C t^{-N/2}\log t$ and
\begin{equation*}
t^{N/2}\Big\|\frac{u(x,t)}{\log t}-C_{A,N} U_\a(x,t)\Big\|_{L^\infty(|x|\le K\sqrt t)}\to0\quad\mbox{as }t\to\infty\quad\forall K>0,
\end{equation*}
where $U_\a$ is the fundamental solution of the heat operator with diffusivity $\a$ and $C_{A,N}$ is a constant that depends only on $A$ and $N$ .
\end{teo}
\begin{proof}
Since $u_0\in L^q(\R^N)$ if $q\alpha> N$, by taking $q=\frac N\alpha\frac1{1-\ep}$, $\ep>0$ and applying \eqref{estimate-W-U} we get,
$$\begin{aligned}
&\Big\|\int W(x-y,t)u_0(y)\,dy-\int U_\a(x-y,t)u_0(y)\,dy\Big\|_{L^\infty(\R^N)}\\
&\le C_q \|u_0\|_q t^{-\frac {N+1}{2q}}= C_\ep t^{-\frac\alpha2(1-\ep)-\frac\alpha{2N}(1-\ep)}\\
&= C_\ep t^{-\frac\alpha2} t^{-\frac\alpha{2N}+\ep\frac\alpha2\big(1+\frac1N\big)}
=C_\mu t^{-\frac\alpha2} t^{-\mu},
\end{aligned}
$$
with $0<\mu<\frac\alpha{2N}$ arbitrary.

Thus,
$$
\|u(\cdot,t)- u_\Delta(\cdot,t)\|_{L^\infty(\R^N)}\le e^{-t}\|u_0\|_{L^\infty(\R^N)}+C_\mu t^{-\frac\alpha2} t^{-\mu},
$$
and \eqref{aprox-homo} holds.

The asymptotic behavior of $u$ follows immediately from \eqref{aprox-homo}
together with \eqref{asym-heat-alpha} and \eqref{heat-N} respectively applied to the solution $u_\Delta$ of the homogeneous heat equation (see \cite{KP,KU}).
\end{proof}

\bigskip

Finally, we prove that $u$ behaves at each positive time as its initial datum.

\begin{prop}\label{u_L}
Let $u$ be the solution of \eqref{homo-sect2} with an initial datum $u_0\in L^\infty$ such that $|x|^\alpha u_0(x)\to A>0$ as $|x|\to\infty$ for some $\alpha>0$. Then,
\begin{enumerate}
\item For every $T>0$ there exists a constant $C_T>0$ such that $|u(x,t)|\le \frac {C_T}{(1+|x|)^\alpha}$ if $0\le t\le T$.

    \medskip

    \item $|x|^\alpha u(x,t)\to A$ as $|x|\to\infty$ uniformly for $t$ bounded.

\end{enumerate}

\end{prop}
\begin{proof} The proof follows from a fixed point argument. In fact, $u$ is a fixed point of the operator
$$
\T v(x,t)=e^{-t}u_0(x)+\int_0^t\int e^{-(t-s)}J(x-y) v(y,s)\,dy\,ds.
$$

Reciprocally, every bounded function that is a fixed point of the operator $\T$ in $\R^N\times(0,t_0)$ is a bounded solution of problem \eqref{homo-sect2}. By uniqueness, it coincides with $u$ in that time interval.

Thus, in order to prove the proposition we will show that $\T$ has a fixed point in the set
$$
\begin{aligned}
&\K:=\{v\in L^\infty(\R^N\times(0,t_0))\,/\,(1+|x|)^\alpha |v(x,t)|\le 2B,
|x|^\alpha v(x,t)\to A\mbox{ as }|x|\to\infty\\
&\hskip1.5cm \mbox{uniformly with respect to }t\in [0,t_0]\}.
\end{aligned}
$$
The  time step $t_0$ will be independent of the constant $B$ where $B$ is a bound of $(1+|x|)^\alpha|u_0(x)|$. Therefore, we can proceed inductively and find a constant for any time interval thus proving (1). Moreover, (2) follows immediately from this argument.

So, let $v\in\K$ and let $w=\T v$. Then, since the support of $J$ is contained in a ball $B_R$ and $|x-y|\le R$ implies that $|x|\le R+|y|$, by assuming that $R\ge1$ we get
$$
\begin{aligned}
&(1+|x|)^\alpha |w(x,t)|\le e^{-t}(1+|x|)^\alpha|u_0(x)|+\int_0^t\int e^{-(t-s)}J(x-y)(1+|x|)^\alpha|v(y,s)|\,dy\,ds\\
&\hskip1.5cm\le e^{-t}B+\int_0^t\int e^{-(t-s)}J(x-y)(1+(R+|y|)
^\alpha)|v(y,s)|\,dy\,ds\\
&\hskip1.5cm \le B+(R+1)^\alpha B t\le 2B.
\end{aligned}
$$
If $t\le t_0=(R+1)^{-\alpha}$.

Now, let us prove that $|x|^\alpha w(x,t)\to A$ as $|x|\to \infty$ uniformly for $t\in (0,t_0)$. In fact, if $|x-y|<R$, there holds that $|y|>|x|-R$ and $$\Big|\frac{|x|}{|y|}-1\Big|<\frac R{|y|}<\frac R{|x|-R}.$$

Thus, if $|x|$ is large enough, $|\frac{|x|^\alpha}{|y|^\alpha}-1\big|\le\ep$.

Therefore, 
\begin{align*}
&\big|\int J(x-y)|x|^\alpha v(y,s)\,dy-A\Big|\le\\
&\hskip2cm \leq \int J(x-y)\big||y|^\alpha v(y,s)-A\Big|\,dy+\Big|\int J(x-y)\big||x|^\alpha-|y|^\alpha|\, |v(y,s)|\,dy\Big|\\
&\hskip2cm \le \int_{|y|>|x|-R}J(x-y)\ep\,dy+\ep\int J(x-y)|y|^\alpha |v(y,s)|\,dy <(1+2B)\ep,
\end{align*}
if $|x|$ is large enough.

Finally,
$$
\begin{aligned}
\big||x|^\alpha w(x,t)-A\big|\le &e^{-t}\big(|x|^\alpha u_0(x)-A\
\big)+
\int_0^te^{-(t-s)}\Big|\int J(x-y)|x|^\alpha v(y,s)\,dy-A\Big|\,ds\\<&2(1+B)\ep.
\end{aligned}
$$

This ends the proof.
\end{proof}

\section{Existence, uniqueness and first properties of the solution}\label{existence}\setcounter{equation}{0}
In this section we prove existence, uniqueness and time local  properties of the solution to the
nonlocal diffusion equation with absorption 
\begin{equation}\begin{aligned}\label{problem-modulo}
&u_t=Lu-|u|^{p-1}u\\
&u(x,0)=u_0(x)
\end{aligned}
\end{equation}
 In particular, we are interested in nonnegative
bounded initial data that behave as a negative power at infinity. That is, which
satisfy \eqref{condalpha}.

Local existence and uniqueness of the solution to \eqref{problem-modulo} with
$p>1$ and bounded initial datum follows by a fixed point argument. For instance, if we call $S(t)$
the semigroup associated to the equation $u_t=Lu$ in $L^\infty(\R^N)$, we can find the solution as
a fixed point in $\K:=\{v\in L^\infty(\R^N\times(0,t_0))\,/\,\|v\|_\infty\le 2\|u_0\|_\infty\}$ of
the operator
$$
\T v=S(t)u_0-\int_0^t S(t-s)|v(x,s)|^{p-1}v(x,s)\,ds.
$$

In fact, it is easy to see that for small $t_0$ the operator $\T$ is a contraction from $\K$ to
$\K$. Thus, there exists a bounded solution in the time interval $(0,t_0)$. Since positive
constants are supersolutions and negative constants are subsolutions to problem \eqref{problem-modulo}, and the function $\tau\to|\tau|^{p-1}\tau$ is locally
Lipschitz, the comparison principle (see, for instance, \cite{LW}) implies that the fixed point
$u$ is bounded by $\|u_0\|_\infty$. Therefore, the solution can be extended for all times.

Moreover, when $u_0\ge0$, the solution $u_L$ of the homogeneous equation \eqref{probl-homo} with initial datum $u_0$ is nonnegative. Thus, $u_L$ is a supersolution and $0$ is a subsolution to \eqref{problem-modulo}. By the comparison principle we deduce that
\begin{equation}\label{compar-u-u_L}
0\le u(x,t)\le u_L(x,t).
\end{equation}

 For the type of initial data we are interested in, much more can be said. In fact,
\begin{teo}\label{local-prop-u}
Let $u_0\in L^\infty(\R^N)$, $u_0\ge0$ be such that $|x|^\alpha u_0(x)\to A>0$ as $|x|\to\infty$ with $\alpha, A>0$. Let $p>1$ and let $u$ be the  solution  to \eqref{problem}. Then,
\begin{enumerate}
\item For every $T<\infty$, there exists a constant $C_T$ such that
$$
u(x,t)\le C_T(1+|x|)^{-\alpha}\quad\mbox{for }t\le T.
$$

    \medskip

    \item  If $\alpha<N$, there exists a constant $C$ such that $u(x,t)\le C t^{-\alpha/2}$. If $\alpha=N$, there exists a constant $C$ such that $u(x,t)\le C t^{-N/2}\log t$.

        \medskip

    \item For every $t>0$, $|x|^\alpha u(x,t)\to A$ as $|x|\to\infty$ uniformly for $t$ in bounded sets.

\end{enumerate}

\end{teo}
\begin{proof}(1) and (2) follow immediately from the estimate \eqref{compar-u-u_L}
above and the results for $u_L$ (Proposition \ref{u_L} (1)).

In order to prove (3) we use the variations of constants formula
$$
u(x,t)=u_L(x,t)-\int_0^t S(t-s)u^p(x,s)\,ds.
$$

We already know that the first term in the right hand side has the correct limit uniformly for $t$ bounded (Proposition \ref{u_L}(2)). Thus, we have to prove that the second term goes to 0 faster than $|x|^{-\alpha}$.

We think of $u^p(x,s)$ as a nonnegative, bounded, initial condition that
satisfies that, for a constant $B$ that is independent of $s\in[0,T]$, there holds that $|x|^{p\alpha}u^p(x,s)\le B$. By
Proposition \ref{u_L} and Theorem \ref{teo-u_L} we know that
$$
|x|^{\alpha p} S(t-s)u^p(x,s)\le C_T\quad\mbox{if }0\le s\le t\le T,
$$
for a certain constant $C_T$. Thus,
$$
|x|^\alpha S(t-s)u^p(x,s)=|x|^{-\alpha(p-1)}|x|^{\alpha p}S(t-s)u^p(x,s)\le C_T |x|^{-\alpha(p-1)}<\ep,
$$
if $|x|$ is large. Therefore,
$$
\Big||x|^\alpha\int_0^tS(t-s)u^p(x,s)\,dx\Big|<\ep T\quad\mbox{if}\quad0\le s\le t\le T\quad\mbox{and }|x|\mbox{ is large}.
$$

\end{proof}
\bigskip

\section{Asymptotic behavior for the equation with absorption}\label{result}
\setcounter{equation}{0}

In this section we prove our main result. Namely, that in the supercritical case, the solution to
\eqref{problem} with a bounded nonnegative initial datum $u_0$ satisfying
\eqref{condalpha} has the same asymptotic behavior as the one of the homogeneous heat equation
with diffusivity $\a$.

\begin{teo}
Let $u_0\ge0$, $u_0\in L^\infty(\R^N)$ be such that $|x|^\alpha u_0(x)\to A>0$ as $|x|\to\infty$ with $0<\alpha\le N$. Let $p>1+\frac2\alpha$. Let $u$ be the solution to \eqref{problem} with $u(x,0)=u_0(x)$.

Then, if $\alpha<N$
$$
t^{\alpha/2}\|u(x,t)-U_{\alpha, A}(x,t)\|_{L^\infty(|x|\le K\sqrt t)}\to0\quad\mbox{as}\quad t\to\infty\quad\forall K>0,
$$
where $U_{\alpha,A}$ is the solution to \eqref{U-alpha}.

If $\alpha=N$,
\begin{equation*}
t^{N/2}\Big\|\frac{u(x,t)}{\log t}-C_{A,N} U_a(x,t)\Big\|_{L^\infty(|x|\le K\sqrt t)}\to0\quad\mbox{as }t\to\infty\quad\forall K>0
\end{equation*}
where $U_a$ is the fundamental solution of the heat operator with diffusivity $\a$ and $C_{A,N}$ is a constant that depends only on $A$ and $N$ .

\end{teo}
\begin{proof}
By Theorem \ref{local-prop-u} we know that $(1+|x|)^\alpha u(x,t_0)
\le C(t_0)$. Thus, by Theorem \ref{teo-u_L} we have that the solution $u_L$ to
\begin{equation}\label{u_L-t_0}
\begin{aligned}
&v_t=Lv\qquad\mbox{in }\R^N\times(t_0,\infty)\\
&v(x,t_0)=u(x,t_0),
\end{aligned}
\end{equation}
satisfies
$$
t^{\alpha/2}\|u_L(x,t)-U_{\alpha, A}(x,t)\|_{L^\infty(|x|\le K\sqrt t)}\to0\quad\mbox{as}\quad t\to\infty\quad\forall K>0 \quad\mbox{if}\quad \alpha<N,
$$
\begin{equation*}
t^{N/2}\Big\|\frac{u_L(x,t)}{\log t}-C_{A,N} U(x,t)\Big\|_{L^\infty(|x|\le K\sqrt t)}\to0\quad\mbox{as }t\to\infty\quad\forall K>0\quad\mbox{if}\quad \alpha=N.
\end{equation*}

Thus, the theorem will be proved if we show that for every $\ep>0$ there exists $t_0$ such that
$$
t^{\alpha/2}\|u(x,t)-u_L(x,t)\|_{L^\infty(\R^N)}\le \ep\qquad\mbox{for }t\ge 2t_0.
$$
Here $u_L$ is the solution of \eqref{u_L-t_0}.

In order to prove this result we need to estimate
$$
u(x,t)-u_L(x,t)=-\int_{t_0}^t S(t-s)u^p(x,s)\,ds.
$$

We begin with the case $\alpha<N$.

Let us estimate the integrand $S(t-s)u^p(x,s)$. On one hand, since by \eqref{compar-u-u_L} and Theorem \ref{teo-u_L}  there holds that $u(x,s)\le C s^{-\alpha/2}$, the maximum principle applied to the solutions of $v_t=Lv$, renders the estimate
\begin{equation}\label{1}
0\le S(t-s)u^p(x,s)\le C s^{-\alpha p/2}.
\end{equation}

On the other hand,
$$
\begin{aligned}
S(t-s)\,u^p(x,s) &= e^{-(t-s)}u^p(x,s)+\int W(x-y,t-s)u^p(y,s)\,dy\\
&\le e^{-(t-s)}s^{-\frac\alpha2 p}+\|W(\cdot,t-s)\|_{q'}\|u^p(\cdot,s)\|_q.
 \end{aligned}
 $$

  As $p>1$ we can take $q=\frac N\alpha$. Then,
$\|W(\cdot,t-s)\|_{q'}\le C(t-s)^{-\frac N{2q}}=C(t-s)^{-\frac\alpha2}$.

On the other hand, $\|u^p(x,s)\|_{L^q(\R^N)}=\|u(x,s)\|_{L^{pq}(\R^N)}^p$.
Thus, since $q=\frac N\alpha>1$,
$$
 \begin{aligned}
 \|u^p(\cdot,s)\|_q^{\frac{1}{p}}&\le \|e^{-s}u_0(\cdot)\|_{pq}
 +\Big\|\int W(x-y,s)u_0(y)\,dy\Big\|_{pq}\\
 &\le C e^{-s}+C[u_0]_{q,\infty}\|W(\cdot,s)\|_r,
 \end{aligned}
 $$
  with $1+\frac1{pq}=\frac1r+\frac1q$ and $[u_0]_{q,\infty}=\sup_{\lambda>0}\Big(\lambda^q
 \big|\{|u_0|>\lambda\}\big|\Big)^{1/q}$ (see, for instance \cite{G}, Theorem 1.4.24).

 Since $u_0$ is bounded and $|x|^\alpha u_0(x)\to A$ as $|x|\to\infty$ there holds that $u_0\in L^{q,\infty}$ and,
 $$
 \|u^p(\cdot,s)\|_q^{\frac{1}{p}}\le
 C e^{-s}+C s^{-\frac N{2r'}}=  C s^{-\frac{\alpha(p-1)}{2p}},
 $$
and hence,
$$\|u^p(\cdot,s)\|_q\le C s^{-\frac{\alpha}{2}(p-1)}.$$ 

Therefore,
\begin{equation}\label{2}
S(t-s)\,u^p(x,s)\le C e^{-(t-s)}s^{-\frac\alpha2 p}+C(t-s)^{-\frac\alpha2}s^{-\frac\alpha2(p-1)}\le C(t-s)^{-\frac\alpha2}s^{-\frac\alpha2(p-1)}.
\end{equation}

\bigskip

Now we estimate the approximation error for $t\ge 2t_0$. We need to separate the
integral and use estimate \eqref{1} for $s\in(t/2,t)$ and estimate \eqref{2} for $s\in(t_0,t/2)$.
We obtain, by using that $\frac\alpha2(p-1)-1\ge0$ (critical or supercritical
cases),
\begin{eqnarray*}
\int_{t_0}^tS(t-s)\,u^p(x,s)\,ds& \leq &
C\int^{t/2}_{t_0}(t-s)^{-\frac\alpha2}s^{-\frac\alpha2(p-1)}\,ds+
C\int_{t/2}^ts^{-\frac\alpha2p}\,ds\\
& \le &( t/2)^{-\frac\alpha2}\int^{t/2}_{t_0} s^{-\frac\alpha2(p-1)}\,ds+ C\int_{t/2}^ts^{-\frac\alpha2p}\,ds\\
&\le &Ct^{-\frac\alpha2} t_0^{-\frac\alpha2(p-1)+1}+Ct^{-\frac\alpha2p+1}.
\end{eqnarray*}

\medskip

Thus, if $t_0$ is large enough and $t\ge 2 t_0$, we get, by using  that we are in the supercritical case $p>1+2/\alpha$,
$$\begin{aligned}
t^{\frac\alpha2}\int_{t_0}^tS(t-s)\,u^p(x,s)\,ds
\le C t_0^{-\frac\alpha2(p-1)+1}
+C t^{-\frac\alpha2(p-1)+1}<\ep.
\end{aligned}
$$

\bigskip

We remark that this is a sharp estimate that shows, in particular, that in the critical case the absorption and the diffusion terms are of the same order.

\bigskip

Now we analyze the case $\alpha=N$.

We proceed as before. Estimate \eqref{1} is changed to
 \begin{equation}\label{1p}
 0\le S(t-s)u^p(x,s)\le C s^{-Np/2}\log^p s.
 \end{equation}

 For the equivalent to estimate \eqref{2} we have to proceed differently since the general form of Young's inequality that we have used is not valid when $q=1$ as would be now the case.

 We take instead $q=\frac1{1-\delta}$ with $\delta>0$ small to be chosen later (we could have
  proceeded in this way before but we wouldn't have gotten the sharp estimate), and use Young's inequality to get
 $$
\|u^p(\cdot,s)\|_{L^1}=\|u(\cdot,s)\|_{L^p}^p\le C e^{-ps}+\|W(\cdot,s)\|_{L^r}^p\|u_0\|_{L^q}^p,
$$\
with $1+\frac1p=\frac1r+\frac1q$. So that, $\frac1{r'}=1-\delta-\frac1p$ and
$\frac{Np}{2r'}=\frac N2(p-1-\delta p)$. Thus,
$$
\|u^p(\cdot,s)\|_{L^1}\le C e^{-ps}+C_\mu s^{-\frac N2(p-1)+\mu} \le
C_\mu s^{-\frac N2(p-1)+\mu},
 $$
 with $\mu=\frac N2\delta p$ as small as needed.

 Therefore,
 \begin{equation}\label{2p}
 \begin{aligned}
 S(t-s)u^p(x,s)&\le e^{-(t-s)}s^{-\frac N2p}\log^p s+\|W(\cdot,t-s)\|_\infty\|u^p(\cdot,s)\|_1\\
 &\le Ce^{-(t-s)}s^{-\frac N2p}\log^p s+C_\mu(t-s)^{-\frac N2} s^{-\frac N2(p-1)+\mu} \\
 & \le C_\mu(t-s)^{-\frac N2} s^{-\frac N2(p-1)+\mu}.
 \end{aligned}
 \end{equation}

 As before, we use \eqref{1p} and \eqref{2p} to estimate the error. We have, by taking $0<\mu<\frac N2(p-1)-1$,
 $$
 \begin{aligned}
 \|u(x,t)-u_L(x,t)\|_{L^\infty(\R^N)}&\le \big\|\int_0^t S(t-s)u^p(x,s)\,ds\|_{L^\infty(\R^N)}\\
 &\le C_\mu\int_{t_0}^{t/2}(t-s)^{-\frac N2} s^{-\frac N2(p-1)+\mu}\,ds+C
 \int_{t/2}^t s^{-\frac N2p}\log^p s\,ds\\
&\le C_\mu\big( t/2\big)^{-\frac N2}t_0^{-\frac N2(p-1)+1+\mu}+Ct^{-\frac N2p+1}\log^p t.
 \end{aligned}
 $$

So that, since $p>1+2/N$, if $t_0$ is large enough and $t\ge2t_0$ we get
$$
t^{N/2}\|u(x,t)-u_L(x,t)\|_{L^\infty(\R^N)}\le C_\mu t_0^{-\frac N2(p-1)+1+\mu}+C t^{-\frac N2(p-1)+1}\log^p t<\ep.
$$

So, the theorem is proved.
\end{proof}

\end{document}